\documentclass[12pt]{article}

\usepackage{amsfonts, amsmath, amssymb,amscd,amsthm}
\usepackage[all]{xy}
\usepackage{latexsym}
\usepackage{amsfonts}
\usepackage{graphicx,color}
\usepackage{verbatim}

\newtheorem{theorem}{Theorem}[section]
\newtheorem{proposition}[theorem]{Proposition}
\newtheorem{lemma}[theorem]{Lemma}
\newtheorem{definition}[theorem]{Definition}
\newtheorem{corollary}[theorem]{Corollary}

\newcommand{\R}{\mathbb{R}}

\newcommand{\vol}{{\rm Vol}}

\title{Boundedness of Laplacian eigenfunctions on manifolds of infinite volume}

\author{Leonardo P. Bonorino, Patr\'icia K. Klaser, Miriam Telichevesky}

\begin{document}

\maketitle

\begin{abstract}
In a Hadamard manifold $M$, it is proved that if $u$ is a $\lambda$-eigenfunction of the Laplacian that belongs to $L^p(M)$ for some $p \ge 2$,  then $u$ is bounded and $\|u\|_{\infty} \le C \|u\|_p,$ where $C$ depends only on $p$, $\lambda$ and on the dimension of $M$. This result is obtained in the more general context of a complete Riemannian manifold endowed with an isoperimetric function $H$ satisfying some integrability condition. In this case, the constant $C$ depends on $p,\lambda$ and $H.$
\end{abstract}

\section{Introduction}

Let $M$ be a complete Riemannian manifold with Laplace-Beltrami operator $\Delta.$ Given an open smooth subset $U\subset M$ and $\lambda\in \mathbb{R}$, we call $\lambda$-eigenfunction of $U$ any nontrivial $u\in C^2(U)$ that satisfies 

\begin{equation}\label{eq} \Delta u + \lambda u = 0\text{ in } U.\end{equation} If in addition $U$ has nonempty boundary, we require that $u$ vanishes continuously on $\partial U$. 

When $U$ is compact, according to the spectral theory for elliptic operators, the numbers $\lambda$ for which \eqref{eq} has a nontrivial classical solution are terms of an increasing unbounded sequence. They are called the eigenvalues of $-\Delta$ and are the only elements of its spectrum.  
If $U$ is noncompact, the situation is more delicate since the spectrum of $-\Delta$ can contain elements that are not eigenvalues. Furthermore, 
a $\lambda$-eigenfunction may not belong to $L^2(U)$ or even to $L^{\infty}(U)$.
This raises the questions whether a $\lambda$-eigenfunction is in some $L^p(U)$ and whether this implies its boundedness. 

In this setting, A. Cianchi and V.G. Maz'ya \cite{ChMa} investigate bounds for $L^2$ eigenfunctions in noncompact manifolds of finite measure. They considered a slightly different eigenvalue problem, which coincides with \eqref{eq} in the case of empty boundary. Under assumptions on the isoperimetric profile of $M$ (see Definition \ref{def do H}), 
the authors obtained the estimate
\begin{equation}\label{eq_estimativaporL2}
\|u\|_{L^{\infty}} \le C \|u\|_{L^2}\end{equation} for any $u$ $\lambda$-eigenfunction of $M$, where $C>0$ is a constant that depends only on the isoperimetric profile and on $\lambda$. Furthermore, their result is sharp in the sense that if $H$ is a suitable isoperimetric profile that does not satisfy the assumptions above mentioned, then there is a manifold $M$ with isoperimetric profile close to $H$
that admits an eigenfunction $u\in L^2(M)\backslash L^\infty(M).$

In the present work we prove that, under the same assumptions on
the isoperimetric profile considered by A. Cianchi and V.G. Maz'ya in
\cite{ChMa}, the estimative \eqref{eq_estimativaporL2} holds with no restriction on the volume of $M$.
Moreover, we also obtain the more general estimative
$\|u\|_{L^{\infty}} \le C \|u\|_{L^p}$ for $\lambda$-eigenfunctions in $L^p$, $p\ge 2.$ 

These new results establishing a bound for the $L^\infty$ norm are useful since for some interesting manifolds with infinite volume there exist $\lambda$-eigenfunctions in $L^2$. This is the case of Hadamard manifolds with curvature going to $-\infty$, for which the existence of eigenfunctions in $L^2$ is proved in \cite{DLi}.
Another application is for manifolds that do not admit eigenfunctions in $L^2,$ but in $L^p$ for $p > 2$. The hyperbolic spaces $\mathbb{H}^n$ are examples of such manifolds, where for $\lambda \le \lambda_1(\mathbb{H}^n)$ the $\lambda$-eigenfunctions belong to $L^P$ for $p>2/\sqrt{1-\lambda/\lambda_1},$ see  \cite{AA}.

\section{Main result}\label{secaoLp}

The statement of our main result requires the definition of isoperimetric functions in manifolds. This concept is a generalisation of isoperimetric profile, which is the largest isoperimetric function of a manifold. 

\begin{definition}\label{def do H}
Consider $M$ a complete Riemannian manifold. An isoperimetric function on $M$ is a function $H:\left[0, |M|\right]\rightarrow \R$ that satisfies 
\begin{equation}\label{eq do H}
H(|\Omega|)\leq |\partial \Omega| \; \forall\; \Omega \subset \subset M,
\end{equation}
where $|\,\cdot\,| $ stands for the Hausdorff measure.
If $M$ has infinite measure, $H$ is defined in $[0,\infty).$
\end{definition}

$H(s)$ gives an lower bound for the measure of the boundary of any set of volume $s.$ Since any Riemannian manifold is locally Euclidean, it is expectable that there exists $H$ such that $$H(s)\approx Cs^{n-1/n},$$ for $s$ close to zero. This is true in compact manifolds, but might fail for noncompact ones. Nevertheless, even a noncompact manifold may admit an isoperimetric function good enough for our result. This is the case when there exists $H$ for which the associated isoperimetric function (a.i.f.), given by $$H_a(t):=\int_0^t \frac{s}{H(s)^2}ds,\;t\in[0,|M|),$$ is well-defined.


\begin{theorem}\label{L2eigenfunctionsInSomeManifoldsAreLinfinite}
 Let $M$ be a complete Riemannian manifold and let $H$ be an isoperimetric  function on $M$ with well-defined a.i.f. $H_a.$ Consider $U\subset M$ be a smooth domain, possibly unbounded, $\lambda>0$ and $p\ge 2$. There exists a constant $C=C(\lambda, p, H)$ such that for all nontrivial solution $w\in L^p(U)$ of
\begin{equation}\label{eq eigenf problem}
\left\lbrace \begin{array}{l}
            -\Delta v=\lambda v \text{ in } U\\[5pt]
             v=0 \text{ on }\partial U\\
             \end{array}\right.
\end{equation}
that belongs to $C^2(U)\cap C^0(\overline{U}),$ it holds
$$\|w\|_{\infty} \le C\|w\|_p. $$ 
Moreover, this constant is given by $C(\lambda, p, H)= 2(H_a^{-1}(\frac{1}{2\lambda}))^{-1/p}$.

\end{theorem}

\subsection{Lemmata}

In this section we prove some lemmas required in the proof of Theorem \ref{L2eigenfunctionsInSomeManifoldsAreLinfinite}. Henceforth $M$ will be a complete Riemannian manifold endowed with an isoperimetric function $H$ that has a well defined a.i.f. $H_a.$ 

\begin{lemma}\label{lap cte gen}
Let $\Omega\subset M$ be a domain with finite measure and $u$ be a solution of $-\Delta u=1 \text{ in }\Omega,$ $u=0 \text{ on }\partial \Omega.$
Then $$\sup u \leq H_a\left(|\Omega|\right).$$ 
\end{lemma}

\begin{proof}

We first consider the case that $\Omega$ is bounded and hence, the solution $u$ must be nonnegative.

\noindent Let $u$ be a solution of $-\Delta u=1 \text{ in }\Omega,$ $u=0 \text{ on }\partial \Omega.$
Consider $\mu$ the distribution function of $u,$ defined as $\mu(t)=|\{x\in \Omega\;|\;u(x)>t\}|.$
It is known that for almost all $t$ \begin{equation}\label{eq derivada da mu}
\mu'(t)=-\int_{\{u=t\}}\frac{1}{|\nabla u|}da_t,
\end{equation} 
where $da_t$ stands for the area element of $\{u=t\}.$

\noindent For all $t>0,$ the set $\{u(x)>t\}$ is at positive distance from the boundary $\partial \Omega.$ It is compactly contained in $\Omega,$ has boundary $\{u=t\}$ with inner normal vector $\frac{\nabla u}{|\nabla u|}$ well defined for almost all $t>0$ by Sard's Theorem.
If the inner normal vector is well defined for $t,$ we apply the Divergence Theorem on the differential equation $-\Delta u=1$, obtaining
\begin{equation}\label{eq teo do dive}
\mu(t)=\int_{\{u>t\}} 1\; dV=  \int_{\{u=t\}} |\nabla u| da_t.
\end{equation}

\noindent Applying the isoperimetric function $H$ in $\mu(t)$ and using Cauchy-Schwarz inequality, we obtain
$$H(\mu(t)) \le |\{u=t\}| \leq \left(\int_{\{u=t\}}\frac{1}{|\nabla u|}da_t\right)^{1/2} \left( \int_{\{u=t\}} |\nabla u| da_t\right)^{1/2}.$$

\noindent Hence, for almost all $t>0,$ expressions \eqref{eq derivada da mu} and \eqref{eq teo do dive} imply
$$H^2(\mu(t))\le -\mu'(t)\mu(t) \text{  and  }
1\leq \frac{\mu(t)(-\mu'(t))}{H^2(\mu(t))}=-\frac{d}{dt}H_a(\mu(t)).$$

\noindent Integrating the above inequality in $[0, \sup u],$ 
$$\sup u\leq -\left(H_a(\mu(\sup u))-H_a(|\Omega|)\right)=H_a(|\Omega|),$$
the proof is complete for bounded $\Omega$.

\

\noindent To the general case, let $u_n$ be the solution to the Dirichlet problem in $\Omega_n$, where $(\Omega_n)$ is an increasing sequence of bounded sets such that $\Omega=\cup \Omega_n$. From the first case, $(u_n)$ is uniformly bounded by $H_a\left(|\Omega|\right)$ and, from the maximum principle, is an increasing sequence. Hence $u_n$ converges to a solution $u$ that is bounded by $H_a(|\Omega|)$, proving the result.
\end{proof}

We remind that the first eigenvalue of the Laplacian operator in a bounded smooth domain $\Omega$, denoted by $\lambda_1(\Omega)$, is the smallest positive real number $\lambda$ for which $\Omega$ admits a nontrivial $\lambda-$eigenfunction. If $U$ is any domain, possibly unbounded, its first eigenvalue is $$\lambda_1(U)=\inf\{\lambda_1(\Omega)\;|\; \Omega\subset \subset U, \; \Omega \; {\rm smooth } \}.$$

If $\Omega$ is a bounded domain in $M$, $\lambda \le \lambda_1(\Omega)$ and $u$ satisfies $\Delta u+\lambda u \le 0$ in $\Omega$ with $u\ge 0$ on $\partial \Omega$, then it is well known that $u\ge 0$ in $\Omega$. Although this property may be false if $\Omega$ is unbounded, its converse is true.

\begin{lemma}\label{lemaNEW} Let $\lambda > 0$ and $u >0$ be a $C^2$ function that satisfies $\Delta u+ \lambda u \le 0$ in some domain $\Omega$ (possibly unbounded) of a Riemannian manifold $M$. Then 
$$ \lambda \le \lambda_1(\Omega). $$ 
\end{lemma}

\begin{proof} Suppose that $\lambda > \lambda_1(\Omega)$. There exists a bounded smooth domain $\Omega_0 \subset \subset \Omega$ with 
$$ \lambda > \lambda_1(\Omega_0) > \lambda_1(\Omega). $$ 
Let $u_0$ be a positive eigenfunction associated to $\lambda_1(\Omega_0)$ in $\Omega_0$. Since $\Omega_0$ is smooth, $u_0 \in C^0(\overline{\Omega}) \cap C^2(\Omega_0)$
and $u_0=0$ on $\partial \Omega_0$. Since $u \in C^2(\Omega)$, $u \in C^2(\overline{\Omega}_0)$.  The positivity of $u$ in $\Omega$ implies that
$$ \alpha := \sup_{\Omega_0} \frac{ u_0}{u}, $$
is finite. Therefore the function $\alpha u-u_0$ defined in $\overline{\Omega}_0$ contradicts the maximum principle since its minimum value is $0$ attained in an interior point of $\Omega_0,$ but $-\Delta (\alpha u-u_0)\geq 0$ in $\Omega_0.$
\end{proof}

The next proposition is a preliminary version of the main result for bounded domains and it is crucial in its proof. Not only it establishes a bound to the maximum of the solution of some Dirichlet problem by its $L^p$ norm, but gives an estimative that does not depend on the diameter, measure or boundary of its domain $\Omega$.
The idea in the proof of Theorem 2.2 is to build a sequence of functions defined in bounded domains that converges to the solution. The uniform boundedness of this sequence is guaranteed by the proposition below. 
We remark that the positive boundary data is necessary,
because there is no $\lambda$-eigenfunction in a bounded domain $\Omega$ for $\lambda < \lambda_1(\Omega)$.

\

The proof follows ideas from an estimative for quotient relating different norms of eigenfunctions of some elliptic operators in domains of $\mathbb{R}^n$ presented in \cite{LeonardoMontenegro} (Section 7). The associated isomperimetric function (a. i. f.) arises to adapt these ideas for manifolds. It is used in Lemma \ref{lap cte gen} to estimate the supremum of a function with constant Laplacian by the measure of its domain.

\begin{proposition}\label{theo autof}
Let $\Omega\subset M$ be a bounded domain and $w \in C^2\left(\Omega\right) \cap C^0\left(\bar{\Omega}\right)$ be a nontrivial solution of
$$\left\lbrace \begin{array}{l}
             -\Delta v=\lambda v \text{ in } \Omega\\[5pt]
             v=\gamma \text{ on }\partial \Omega\\
             \end{array}\right.$$
for some $0 < \lambda  \le \lambda_1(\Omega)$ and $\gamma \ge 0$. Then
\begin{equation}\label{estimativa Lifnito}
\|w\|_p^p\geq\left(\frac{\|w\|_\infty + \gamma}{2}\right)^p {H_a}^{-1}\left(\frac{\|w\|_\infty-\gamma}{2\lambda \|w\|_\infty }\right).
\end{equation}

\end{proposition}

\begin{proof}
Denote by $K=\|w\|_\infty$. Notice that $w$ cannot change sign, because $\lambda \le \lambda_1(\Omega)$. We therefore assume without loss of generality that $w\ge 0$.

\noindent Let 
$$\widetilde{\Omega}=\left\{x \in \Omega \; | \; w(x)>\frac{K + \gamma}{2}\right\}.$$
Then,
\begin{equation}\label{norma lr}
\|w\|_p^p=\int_\Omega |w|^p dV \geq \int_{\widetilde{\Omega}} |w|^p dV\geq \left(\frac{K + \gamma}{2}\right)^p |\widetilde{\Omega}|
\end{equation}
On the other hand,
$$-\Delta w=\lambda w \leq \lambda K.$$
By the comparison principle, $w\le u$ in $\widetilde{\Omega}$ where $u$ is solution of
$$\left\lbrace \begin{array}{l}
             -\Delta v=\lambda K \text{ in }\widetilde{\Omega}\\[5pt]
             v=\displaystyle \frac{K + \gamma}{2} \text{ on }\partial \widetilde{\Omega}\\
             \end{array}\right.$$
Lemma \ref{lap cte gen} gives an upper bound for $u$
$$K \le\sup u \leq \frac{K + \gamma}{2} + \lambda K H_a\left(|\widetilde{\Omega}|\right).$$
Hence,
$$|\widetilde{\Omega}|\geq {H_a}^{-1}\left(\frac{K-\gamma}{2 \lambda K }\right).$$
Therefore, from inequality \eqref{norma lr}, we obtain
$$\|w\|_p^p\geq\left(\frac{K + \gamma}{2}\right)^p {H_a}^{-1}\left(\frac{K-\gamma}{2\lambda K }\right).$$

\end{proof}

\

\noindent {\bf Remark:} For the case $\gamma=0$, the result is true for any $\lambda > 0,$ because we may assume $\max w=\max |w|.$

~\\

\begin{lemma}
Let $\Omega \subset M$ be a domain with finite volume, possibly unbounded, and $u \in C^2(\overline{\Omega})$ be a classical solution of
$$ \left\{ \begin{array}{rcl} -\Delta u & = & \lambda u  \quad {\rm in } \quad \Omega \\[5pt]
                              u & > & \gamma \quad {\rm in} \quad \Omega \\[5pt]
                              u & = & \gamma \quad {\rm on} \quad \partial \Omega, \end{array} \right.
$$                              
where $\lambda > 0$ and $\gamma > 0$. If $u \in L^p(\Omega)$ for $p \ge 2$, then $u \in H^1(\Omega)$
and
$$ \int_{\Omega} |\nabla u|^2 dV \le \lambda \int_{\Omega} u^2 dV .$$
\label{whoisinLpisinH1} 
\end{lemma}

\begin{proof}
For $\varepsilon>0$, let $\psi=\psi_{\varepsilon}$ be defined in $\Omega$ by $\psi(x)=u(x)-\gamma-\varepsilon$. Let $A=A_\varepsilon=\{\psi>0\}=\{u>\gamma+\varepsilon\}.$
We claim that $u\in H^1(A)$ and $$\int_{A} |\nabla u|^2 \, dV \le  \lambda \int_A  u^2 \; dV.$$
Hence, letting $\varepsilon$ go to zero, the result holds.

\

\noindent In order to prove the claim, fix $o\in M$ and, for each $R>1,$ let $\eta=\eta_R:M\rightarrow \mathbb{R}$ be a smooth radial function satisfying: $\eta_R (x)\in[0,1], \; \eta_R(x)= 1 \text{ if }x\in  B_{R-1}(o),\; \eta_R(x)=0 \text{ if }x\notin B_R(o) \text{ and } |\nabla \eta_R (x)|<2.$
 
\noindent Let $w=w_{R}$ be defined in $\Omega$ by  $w(x)= \eta^2(x) \psi(x) = \eta^2(x)(u(x)-\gamma-\varepsilon)$. Since $u, \eta \in C^2(\overline{\Omega}),$ we have $w \in H^1(\Omega \cap B_{R+1}).$
Hence, the positive part of $w$, $w^+= \max \{ w, 0\},$ is in $H^1(\Omega \cap B_{R+1})$.
Besides, $w^+ = 0$ in some neighborhood of $\partial (\Omega \cap B_{R+1})$ because $\psi \le - \varepsilon $ on $\partial \Omega $ and $\eta = 0 $ on $B_{R+1}\setminus B_R$.
This implies $w^+\in H_0^1(\Omega\cap B_{R+1})$ and
 $$ \int_{\Omega} \nabla w^+ \nabla u \, dV = \lambda \int_{\Omega} w^+ u \; dV$$
since $u$ is a weak solution. 

\noindent Observe that $w^+ = [\eta^2 \psi]^{+} = \eta^2 \psi^{+}$ and 
$\nabla w^+ = [\eta^2 \nabla u  + (u- \gamma - \varepsilon) 2\eta \nabla \eta] \chi_A$ almost everywhere in $\Omega$.
Therefore
$$ \int_A \eta^2 |\nabla u|^2 + 2\psi \eta \nabla \eta \nabla u \, dV = \lambda \int_A \eta^2 \psi u \; dV $$
and, from Cauchy-Schwarz inequality, 
\begin{equation} \int_{A} \eta^2 |\nabla u|^2 \, dV \le \lambda \int_A \eta^2 \psi u \; dV + 2\left( \int_A \eta^2 |\nabla u|^2 dV \right)^{1/2} \left( \int_A \psi^2 |\nabla \eta|^2 dV \right)^{1/2} .\label{estimateForGradu1}
\end{equation}
Therefore, using that $|\eta| \le 1$, $|\nabla \eta|\le 2$, and $|\psi| \le u$ in $A$, we get
$$ \int_{A} \eta^2 |\nabla u|^2 \, dV \le \lambda \int_A  u^2 \; dV + 4\left( \int_A \eta^2 |\nabla u|^2 dV \right)^{1/2} \left( \int_A u^2  dV \right)^{1/2} .$$
\noindent The integral $\displaystyle \int_A u^2 dV$ is finite since $u \in L^p(\Omega)$ for $p \ge 2$, $A \subset \Omega$ and $|\Omega| < \infty$.
Hence, for any $R > 0$, the corresponding $\eta=\eta_R$ satisfies
$$ \int_{A} \eta_R^2 |\nabla u|^2 \, dV \le  \left(2+ \sqrt{4 + \lambda}\right)^2  \int_A u^2 dV < \infty .$$
Since $\eta_R \to 1$ as $R \to \infty$, it follows that $ \displaystyle \int_{A} |\nabla u|^2 \, dV < \infty$.
Moreover, $u \in L^2(A)$ implies that
$$ \int_A u^2 |\nabla \eta|^2 dV \le 4 \int_{A\backslash B_R} u^2  dV \to 0 \quad {\rm as } \quad R \to \infty,$$
then making $R \to \infty$ in \eqref{estimateForGradu1}, we get
$$\int_{A} |\nabla u|^2 \, dV \le \lambda \int_A \psi u \; dV \le  \lambda \int_A  u^2 \; dV,$$ concluding the proof.

\end{proof}

The next lemma is some uniqueness result based on one established by Brezis and Oswald \cite{BO}.

\begin{lemma}\label{uniquenessAux}  
Let $\Omega \subset M$ be a domain with finite measure, possibly unbounded. If $u_1, u_2 \in H^1(\Omega)\cap C^2(\Omega)\cap C^0(\overline{\Omega})$ satisfying $0 < u_1 \le u_2$ are classical solutions of
$$ \left\{ \begin{array}{rcl} -\Delta u & = & au+b \quad {\rm in} \quad \Omega \\[5pt]
                              u & = & 0 \quad {\rm on} \quad \partial \Omega, \end{array} \right. $$
for $a,\;b$ positive constants, then $u_1 = u_2$.

\end{lemma}

\begin{proof}
Since $u_1, u_2 \in H^1(\Omega) \cap C^2(\Omega)\cap C^0(\overline{\Omega})$ and $u_1=u_2=0$ on $\partial \Omega$,
we can prove that $u_1, u_2 \in H^1_0(\Omega)$. Hence, from the definition of weak solution,
$$ \int_{\Omega} \nabla u_1 \nabla u_2 \, dV = \int_{\Omega} (au_1+b) u_2 \, dV  \quad  {\rm and } \quad \int_{\Omega} \nabla u_2 \nabla u_1 \, dV = \int_{\Omega} (au_2+b) u_1 \, dV .$$
Thus 
$$ \int_{\Omega} \left( \frac{au_1+b}{ u_1} - \frac{a u_2+b}{ u_2} \right) u_1 u_2 \, dV = 0. $$
It is then clear that the integrand is nonnegative, which implies that it is equal to zero. Hence
$$ \frac{au_1+b}{u_1} = \frac{au_2+b}{u_2} ,$$
and, therefore, $u_1=u_2$ completing the proof.
\end{proof}

\noindent {\bf Remark:}
Lemma \ref{uniquenessAux} holds in a more general setting: if $\beta \ge 0$ and $f:\mathbb{R} \to \mathbb{R}$ is a Lipschitz function such that $f(t+\beta)/t$ is decreasing, then the same conclusion holds considering the problem $$ \left\{ \begin{array}{rcl} -\Delta u & = & f(u) \quad {\rm in} \quad \Omega \\[5pt]
                              u & = & \beta \quad {\rm on} \quad \partial \Omega. \end{array} \right. $$

\subsection{Proof of Theorem \ref{L2eigenfunctionsInSomeManifoldsAreLinfinite}}

We assume that $w >0$ in $U$, otherwise we split $U$ into two domains. Fix a point $o \in U,$ let $\gamma_0 = w(o)/2$ and, for $0 < \gamma \le \gamma_0$, consider the set $\Omega =\{ x \in U \, : \, w(x) > \gamma \}.$ Then $\Omega$ is not empty, $\overline{\Omega}\subset U$ and it has finite measure since
$$ \gamma^p | \Omega | \le \int_{\Omega} |w|^p \; dV < \infty. $$
Moreover, since $\Delta w + \lambda w = 0$ and $w$ is positive in $\Omega$, it follows from Lemma \ref{lemaNEW} that $\lambda \le \lambda_1(\Omega)$.
Now define $\Omega_k= \Omega \cap B_k(o)$ for $k \in \mathbb{N}$ and let $z_k \in H_0^1(\Omega_k)$ be the weak solution of 
\begin{equation}
\left\lbrace \begin{array}{rcll}
             -\Delta v & = &\lambda v + \lambda \gamma & \text{ in } \Omega_k\\[5pt]
             v & = & 0 &  \text{ on }\partial \Omega_k \\
             \end{array}\right.
\label{AuxDProblemZk}
\end{equation}If $\Omega_k=\Omega$, $z_k = w-\gamma$. Otherwise,
the existence of solution to this problem is a consequence of $\lambda \le \lambda_1(\Omega) < \lambda(\Omega_k)$ and the classical theory for eigenvalue problems in PDE. Observe that $w_k:=z_k + \gamma$ is a weak solution of 
$$\left\lbrace \begin{array}{rcll}
             -\Delta v & = &\lambda v  & \text{ in } \Omega_k\\[5pt]
             v & = & \gamma &  \text{ on }\partial \Omega_k \\
             \end{array}\right.$$
and that $w_k = \gamma \le w$ on $\partial \Omega_k$. Then $w_k < w$ in $\Omega_k$ since $\lambda \le \lambda_1(\Omega) \le \lambda_1 (\Omega_k)$. The equality $\lambda_1(\Omega)=\lambda_1(\Omega_k)$ only happens if $\Omega=\Omega_k$.
By the same argument,  $w_k \, \ge \, w_m > \gamma$ for any $k>m$.
Hence from Proposition \ref{theo autof},
for each $k\in\mathbb{N},$
\begin{equation}
\left(\frac{\|w_k\|_{\infty} + \gamma}{2}\right)^p {H_a}^{-1}\left(\frac{\|w_k\|_{\infty}-\gamma}{2\lambda \|w_k\|_{\infty} }\right) \le \|w_k\|_p^p \le \|w\|_p^p.
\label{LInfiniteEstimateForw_k}
\end{equation}
Since ${H_a}^{-1}$ is an increasing positive function, the left-hand side diverges to infinity if $\|w_k\|_{\infty}$ goes to infinity.
Hence $(\|w_k\|_{\infty})$ is a bounded sequence and, since it is also increasing, it converges pointwise to some bounded function $\bar{w}$ defined on $\Omega$. Moreover $w_k \le \bar{w} \le w$. To prove that $w$ is bounded, it is sufficient to show that $w=\bar{w}$ in $H^1\left(\Omega\right).$ We observe that since $\Omega$ has finite measure and $w\in L^{p}(\Omega),$ $p\ge 2,$ Lemma \ref{whoisinLpisinH1} implies that $w\in H^1(\Omega).$

\

\noindent Let $z=w-\gamma$ and $\bar{z}=\bar{w} -\gamma.$ Then $z_k \le \bar{z} \le z$ and $z_k$ converges pointwise to $\bar{z}$. We have to show that $\bar{z}=z$. The idea is to verify that $z, \bar{z} \in H^1(\Omega)\cap C^2(\Omega)\cap C^0(\overline{\Omega})$ and satisfy the same Dirichlet problem. 

\noindent First, notice that $z \in C^2(\overline{\Omega})$ since $z=w-\gamma$, $w \in C^2(U)$ and $\overline{\Omega} \subset U$. Besides $z\in H^1(\Omega)$ because $0 \le z =w-\gamma \le w$ and $w \in H^1(U)$. Moreover $z$ is a solution of 
\begin{equation}
\left\{ \begin{array}{rcll}
             -\Delta v & = &\lambda v + \lambda \gamma & \text{ in } \Omega\\[5pt]
             v & = & 0 &  \text{ on }\partial \Omega. \\
             \end{array}\right. 
\label{DProblemAux1} 
\end{equation} 
We now prove similar results for $\bar{z}$. Since $z_k$ is solution of \eqref{AuxDProblemZk}, $0 \le w_k-\gamma = z_k \le w_k \le w$ and $\Omega_k \subset \Omega$, it follows that
$$ \int_{\Omega_k} |\nabla z_k|^2 dV = \lambda \int_{\Omega_k} z_k^2 \; dV + \lambda \gamma \int_{\Omega_k} z_k \; dV \le \lambda \int_{\Omega} w^2 \; dV + \lambda \gamma \int_{\Omega} w \; dV< \infty.$$
Hence the sequence $(\|\nabla z_k\|_2)$ is bounded.  Moreover, $z_k \le w$ implies that $(\|z_k\|_2)$ is bounded. Therefore, up to some subsequence, $z_k$ converges weakly to some function in $H^1_{0}(\Omega)$. This limit is $\bar{z}$, since $z_k$ converges pointwise to $\bar{z}$. Thus $\bar{z}$ is a weak solution of \eqref{DProblemAux1}. Then $\bar{z}$ is a classical solution and it is of class $C^2(\Omega)$. The continuity of $\bar{z}$ on $\partial \Omega$ is a consequence of $0 \le \bar{z} \le z$ and of the fact that $z$ vanishes continuously on $\partial \Omega.$

\noindent Therefore, $z$ and $\bar{z}$ satisfy the hypotheses of Lemma \ref{uniquenessAux}, which implies uniqueness of solution of problem \eqref{DProblemAux1}. Hence $z=\bar{z}$ and $w=\bar{w}$ in $\Omega,$ proving that $w$ is bounded and $\|w\|_{L^{\infty}(\Omega)}=\|\bar{w}\|_{L^{\infty}(\Omega)}$.

\noindent Furthermore, since $w_k \to \bar{w}$ and $w_k \le \bar{w}$, we have 
$$\|w_k\|_{L^{\infty}(\Omega)} \to \|\bar{w}\|_{L^{\infty}(\Omega)} = \|w\|_{L^{\infty}(\Omega)}.$$ 
This implies that \eqref{LInfiniteEstimateForw_k} also holds replacing $w_k$ by $w$. Observe also that $\gamma > 0$ can be chosen so small as we want and then it can be omitted in \eqref{LInfiniteEstimateForw_k}, obtaining
$$\left(\frac{\|w\|_{\infty}}{2}\right)^p {H_a}^{-1}\left(\frac{1}{2\lambda}\right) \le  \|w\|_{L^p(\Omega)}^p,$$
completing the proof. \\

\hfill q.e.d.

\section{Application to Hadamard manifolds}

Hadamard manifolds admit some isoperimetric function with well-defined a.i.f. and, therefore, Theorem \ref{L2eigenfunctionsInSomeManifoldsAreLinfinite} applies to them. Indeed it is a consequence of
Theorem \ref{teo Hadamard tem aif} proved by Christopher B. Croke \cite{Cr1} for $n \ge 3$. For $n=2$ the theorem also holds even for more general manifolds according to E. F. Beckenbach and T. Rad\'o \cite{BR}.

\begin{theorem}\label{teo Hadamard tem aif} Let $N^n$ be a compact Riemmanian manifold (with boundary) of nonpositive sectional curvature.
Suppose that any geodesic ray in $N$ minimizes length up to the point it hits the boundary. Then there exists a positive constant $D(n)$, that depends only on $n$, such that
$$ \vol(\partial N) \ge D(n) (\vol(N))^{1-1/n}.$$
\end{theorem}

If $M$ is a Hadamard manifold, any smooth compact subdomain satisfies the hypotheses of the theorem.  Hence $H(s)=D(n) s^{1-1/n}$
is an isoperimetric function on $M$ and 
$$ H_a(t)= \int_0^t \frac{s}{H(s)^2}ds = \frac{n}{2(D(n))^2}t^{2/n}.$$

\begin{corollary}\label{hadamard l2 implica limitada} Suppose that $M$ is a Hadamard manifold, $U$ is an unbounded domain of $M$ and $w$ is an eigenfunction (solution of \eqref{eq eigenf problem}) in $U$ associated to $\lambda>0.$
If $w \in L^p(U)$ for some $p \ge 2$, then $w$ is bounded and 
$$ \|w\|_{L^{\infty}(U)} \le \frac{2 \,(n \lambda)^{n/2p}} {\left(D(n)\right)^{n/p}} \|w\|_{L^p(U)}.$$
\end{corollary}

In \cite{DLi}, Donnelly and Li proved the following theorem about the existence of $L^2$ eigenfunctions on some manifolds with infinite volume. 
According to Corollary \ref{hadamard l2 implica limitada}, these eigenfunctions are bounded.

\begin{theorem}[Theorem 1.1 of \cite{DLi}] Let $M$ be a complete simply connected negatively curved Riemannian manifold. Fix $p\in M$ and write $K(r)=\sup\{K(x,\pi)\,|\,d(x,p)\ge r\},$ where $K$ is the sectional curvature of $M$ and $\pi$ denotes a $2-$plane in $T_xM.$ If $\lim K(r)=-\infty$ as $r \to \infty$, then $\Delta$ has pure point spectrum.
\end{theorem}

\bigskip

INSTITUTO DE MATEM\'ATICA

Universidade Federal do Rio Grande do Sul, Brasil\\

Leonardo Prange Bonorino

e-mail address: bonorino@mat.ufrgs.br\\

Patr\'icia Kruse Klaser

e-mail address: patricia.klaser@ufrgs.br\\

Miriam Telichevesky

e-mail address: miriam.telichevesky@ufrgs.br\\

\end{document}